\documentclass[draft]{amsart}

\title[Superspecial varieties and modular forms]{Superspecial Abelian Varieties and the Eichler Basis Problem for Hilbert Modular Forms}
\author{Marc-Hubert Nicole}

\usepackage{amssymb}
\usepackage{graphicx, epsfig}
\usepackage{setspace}

\usepackage{
latexsym, amsfonts, amscd, amsmath}

\newtheorem{thm}{Theorem}[section]
\newtheorem{prop}[thm]{Proposition}
\newtheorem{lem}[thm]{Lemma}
\newtheorem{cor}[thm]{Corollary}

\newtheorem{fact}[thm]{Fact}

\newtheorem{dfn}[thm]{Definition}

\newcommand{\End}{{\operatorname{End}}}
\newcommand{\Fr}{{\operatorname{Fr }}}
\newcommand{\Hom}{{\operatorname{Hom}}}

\newcommand{\Norm}{{\operatorname{Norm }}}

\newcommand{\Spec}{{\operatorname{Spec }}}

\newcommand{\Tr}{{\operatorname{Tr }}}

\newcommand{\SL}{{\operatorname{SL }}}
\newcommand{\GL}{{\operatorname{GL}}}
\newcommand{\Gal}{{\operatorname{Gal}}}

\newcommand{\Lie}{{\operatorname{Lie}}}

\newcommand{\gerp}{{\frak{p}}}
\newcommand{\gerq}{{\frak{q}}}

\newcommand{\gerA}{{\frak{A}}}

\newcommand{\gerN}{{\frak{N}}}

\newcommand{\calD}{{\mathcal{D}}}

\newcommand{\calH}{{\mathcal{H}}}

\newcommand{\calM}{{\mathcal{M}}}
\newcommand{\calN}{{\mathcal{N}}}
\newcommand{\calO}{{\mathcal{O}}}

\def\A{\mathbb{A}}
\def\C{\mathbb{C}}
\def\D{\mathbb{D}}

\def\F{\mathbb{F}}

\def\N{\mathbb{N}}

\def\Q{\mathbb{Q}}

\def\Z{\mathbb{Z}}

\def\Fp{\overline{\mathbb{F}}_p}

\newcommand{\Cent}{{{\operatorname{\bf Cent}}}}

\newcommand{\ilim}{{\underset{\longleftarrow}{\lim} \; \; }}
\newcommand{\arr}{{\; \longrightarrow \;}}

\newcommand{\ol}{{\mathcal{O}_L}}

\def\mat#1#2#3#4{\left( \begin{array}{cc} #1 & #2 \\ #3 & #4 \end{array} \right) }

\def\2vector#1#2{\left( \begin{smallmatrix} #1 \\ #2 \end{smallmatrix}
\right)}
\def\deb{ \begin{equation} }
\def\fin{ \end{equation} }

\usepackage[usenames]{color}
\definecolor{Indigo}{rgb}{0.2,0.1,0.7}
\definecolor{Violet}{rgb}{0.5,0.1,0.7}
\definecolor{White}{rgb}{1,1,1}
\definecolor{Green}{rgb}{0.1,0.9,0.2}

\begin{document}

\maketitle


\medskip\noindent
{\bf ABSTRACT.} Let $p$ be an unramified prime in a totally real
field $L$ such that $h^+(L)=1$. Our main result shows that Hilbert
modular newforms of parallel weight two for $\Gamma_0(p)$ can be
constructed naturally, via classical theta series, from modules of
isogenies of superspecial abelian varieties with real
multiplication on a Hilbert moduli space. This can be viewed as a
geometric reinterpretation of the Eichler Basis Problem for
Hilbert modular forms.



\section{Introduction}

Let $L$ be a totally real number field of degree $[L:\Q]=g$, with
$\calO_L$ its ring of integers. This paper is the first of a
series of two aiming at the generalization of classical theorems
of Deuring and Eichler concerning supersingular elliptic curves
and modular forms to the setting of Hilbert modular varieties. In
this work, we suppose throughout that $p$ is unramified in $L$, to
stay closest to the classical line of thought. If $p$ is ramified,
new phenomena occur and interesting connections can be made with
classical works of Manin and Doi-Naganuma (see
\cite{NicoleRamified}).

\noindent We recall the classical case, for $L = \Q$. Let $H$ be
the class number of $B_{p,\infty}$ i.e., the number of left ideal
classes of a maximal order in the rational quaternion algebra
$B_{p,\infty}$ ramified at $p$ and $\infty$. Let $I_i, I_j, 1 \leq
i,j \leq H$ be left ideal classes representatives. Using the norm
of the quaternion algebra, we can define:
$$Q_{ij}(x) := \Norm(x)/\Norm(I_j^{-1}I_i), \mbox{ for } x \in I_j^{-1} I_i, $$
i.e., a quadratic form of level $p$, discriminant $p^2$, with
values in $\N$. Since the quaternion algebra $B_{p, \infty}$ is
definite (i.e., ramified at the infinite place), the
representation numbers $a(n) := |\{ x | Q_{ij}(x) = n \}|$ are
finite. The theta series
$$\theta_{ij}(z) := \sum_{n \in \N} a(n) q^n, \quad \mbox{ for } q = e^{2 \pi i z}, $$
is a modular form of weight $2$ for $\Gamma_0(p) := \{ (
\begin{smallmatrix} a & b \\ c & d
\end{smallmatrix} ) \in \SL_2(\Z) | ( \begin{smallmatrix} a & b \\ c & d \end{smallmatrix} ) =
(\begin{smallmatrix} \ast & \ast \\ 0 & \ast \end{smallmatrix} )
\mod{p} \}$ by the Poisson summation formula. In 1954, Eichler
showed that the $H(H-1)$ cusp forms
$$\theta_{ij}(z) - \theta_{1j}(z), \quad 2 \leq i \leq H, 1 \leq j \leq H, $$ {\em
span} the vector space $S_2 (\Gamma_0(p))$ of cusp forms of weight
$2$ for the group $\Gamma_0(p)$.
 \noindent In 1941, Deuring determined, for $E$ a supersingular
elliptic curve over $\Fp$, that $\End_{\overline{\F}_p}(E)$ is
isomorphic to a maximal order in the quaternion algebra
$B_{p,\infty}$ over $\Q$. \noindent Using e.g., the
$\gerA$-transform of Serre, one can show that there is a bijection
between left ideal classes $[I_1], \dots, [I_H]$ of
$\End_{\overline{\F}_p}(E)$ and isomorphism classes of
supersingular elliptic curves $E_1, \dots, E_H$ over
$\overline{\F}_p$, given functorially by the tensor map
$$[I_i] \mapsto [E \otimes_{\End_{\Fp}(E)} I_i],$$
where the brackets $[-]$ represent the respective isomorphism
classes.

\noindent We are now in position to give a geometric
interpretation of Eichler's original Basis Problem in terms of
modules of isogenies of supersingular elliptic curves.

\begin{prop}
The span of the theta series coming from the modules
$$\Hom(E_i,E_j) \cong I_j^{-1}I_i$$ equipped with the quadratic
degree map, includes the vector space $S_2(\Gamma_0(p))$.
\end{prop}

\noindent
 The bulk of the paper deals with the generalization
of the above geometric interpretation to Hilbert modular forms, by
using modules of isogenies of superspecial points on the Hilbert
moduli space of dimension $g$, endowed naturally with a quadratic
structure via a so-called $\calO_L$-degree map coming from the
polarizations. Recall that for a base scheme $S$, the Hilbert
moduli space classifies abelian varieties $A/S$ with real
multiplication by $\calO_L$ satisfying the Rapoport condition. Let
$k$ be an algebraically closed field $k$ of characteristic $p$. An
abelian variety $A/k$ is \emph{superspecial} if and only if $A
\cong_{/k} E^g$, for $E$ some supersingular elliptic curve defined
over $k$. By a theorem of Deligne (see \cite[Thm. 3.5]{Shioda}),
there is a unique superspecial abelian variety for $g \geq 2$
(this is of course false in general for $g=1$), and it follows
that the superspecial locus on the Hilbert moduli space is finite.

\noindent We state our geometric reinterpretation of the classical
Eichler Basis Problem for Hilbert modular forms, in terms of
$\calO_L$-modules of $\calO_L$-isogenies of superspecial abelian
varieties with $\calO_L$-action succinctly (it follows from
results of Section \ref{leftidealclasses} and Cor.
\ref{basisEichler}):
\begin{thm} Let $h^+(L)=1$, and $p$ unramified in $L$. Let $A_m, A_n$ run through all superspecial points of the
Hilbert moduli space over $\Fp$. The quadratic modules
$\Hom_{\calO_L}(A_m,A_n)$ equipped with the $\calO_L$-degree map
give all quadratic forms of level $p$ attached to $B_{p,\infty}
\otimes L$. It follows that the space $S^{new}_{2}(\Gamma_0((p)))$
is contained in the span of the theta series stemming from the
quadratic modules $\Hom_{\calO_L}(A_m,A_n)$.
\end{thm}

\noindent We streamline and improve slightly upon the results of
\cite{these}. In particular, we demonstrate that this geometric
reinterpretation of the Eichler Basis Problem in terms of abelian
varieties is possible even for levels not found on the Hilbert
moduli space. That is, Cor. \ref{levelonetheta} and Cor.
\ref{basisEichler} imply that the space $S_{2}(\Gamma_0(1))$ is
contained in the span of theta series stemming from quadratic
modules $\Hom_{\calO_L}(A,A')$, for suitable abelian varieties
$A,A'$ with RM with maximal endomorphism orders.

The paper is organized as follows. In Section 2, we prove a
variant of Tate's theorem for supersingular abelian varieties with
real multiplication. In Section 3, we review some well-known
material on Dieudonn\'e modules and give an ad hoc proof of a very
special case of the classification theorem for those modules. In
Section 4, we compute the orders of endomorphisms of the
corresponding abelian varieties which turn out to be superspecial.
In Section 5, we explain the link between the arithmetic of
hereditary orders in quaternion algebras and the geometry of
abelian varieties with real multiplication. In Section 6, we
recall how the Eichler Basis Problem for Hilbert modular forms of
squarefree level follows from the Jacquet-Langlands
correspondence.

\section{A variant of a theorem of Tate} Let $p$ be a rational
prime and let $B_{p, \infty}$ be the rational quaternion algebra
ramified only at $p$ and $\infty$. Let $L$ be a totally real field
of degree $g$ over $\Q$. We suppose throughout that $p$ is
\emph{unramified} in $L$. We let $B_{p, L}$ denote the quaternion
algebra $B_{p, \infty} \otimes_\Q L$ over~$L$.

Let $A$ be an abelian variety over a perfect field $k$ of
characteristic $p$. For a rational prime $\ell$, we let
$$T_{\ell}(A) =
\begin{cases} \ilim A[\ell^n] & \ell \neq p
\\ \D(A) & \ell = p. \end{cases}$$
Here $\D(A)$ denotes the covariant Dieudonn\'e module. Recall that
a Dieudonn\'e module is a left $W(k)[F, V]$-module, free of finite
rank over the Witt vectors $W(k)$ such that $\D/F\D$ has finite
length over $W(k)$. In fact, for $\D(A)$ this length is precisely
$\dim(A)$. Recall also that the operators $F, V$ satisfy the
relations $FV = VF = p, F\lambda = \lambda^p F, \lambda V = V
\lambda^p$ for $\lambda \in W(k)$. The morphisms between
$T_\ell(A_1)$ and $T_\ell(A_2)$ are always $\Z_\ell$-linear if
$\ell \neq p$, and $W(k)[F, V]$-linear if $\ell = p$. If $A$ is
defined over a scheme $S$, we use the notation $\End(A)$ to denote
its endomorphisms over $S$; we emphasize this by using the
notation $\End(A/S)$. The notation $\End^0(A)$ stands for $\End(A)
\otimes_\Z \Q$.

\ \noindent We shall consider the following type of abelian
variety $A$ with additional endomorphism structure:

{\bf Real Multiplication (RM):} The abelian variety $A$ over $\Fp$
is supersingular, and is equipped with RM by $\iota: \calO_L
\hookrightarrow \End(A)$. \noindent By \cite[Lemma 6]{Chai}, the
centralizer $\Cent_{\End^0(A)}(L) \cong B_{p,L}$, and thus
$\Cent_{\End(A)}(\calO_L)$ is isomorphic to an order in $B_{p,L}$.

\noindent Further on, we will impose stringent requirements on the
abelian variety $(A,\iota)$: in particular, $A$ will be
superspecial and the Eichler order $\End_{\calO_L}(A)$ will be of
prescribed level. Otherwise explicitly mentioned, we consider
abelian varieties that satisfy Rapoport's condition (see
\cite{Rapoport}) or equivalently (since $p$ is unramified, see
\cite[Cor. 2.9]{DP}), the Deligne-Pappas condition (cf.
\cite{DP}). Recall that those are moduli conditions: $(A,\iota)$
over $S$ satisfies the Rapoport condition if the Lie algebra of
$A$ is a locally free $\calO_L \otimes_{\Z} \calO_S$-module of
rank $1$. We recall that the Hilbert moduli space $\calM$
decomposes as: $\calM = \bigsqcup_{(\gerA, \gerA^+) \in Cl^+(L)}
\calM_{(\gerA, \gerA^+)}$, where $(\gerA,\gerA^+)$ is an
$\calO_L$-module with a notion of positivity. Under the hypothesis
that the narrow class number $h^+(L)$ is $1$, $\calM =
\calM_{(\calO_L,\calO_L^+)}$, the component of principally
polarized points (see \cite[\S 2.6]{DP}). Denote by $A^t$ the dual
abelian variety. Recall that a principal polarisation $\lambda$ of
an abelian scheme $A$ is a symmetric isomorphism: $\lambda: A
\overset{\cong}{\arr} A^t$, coming from an ample line bundle on
geometric points; we require the polarisation to be
$\calO_L$-linear. Denote by $\Hom_{\calO_L}(A,A^t)^{sym}$ the
invertible $\calO_L$-module of  symmetric $\calO_L$-linear
homomorphisms from $A$ to $A^t$. If $A$ over a field $k$ satisfies
the Rapoport condition and $\Hom_{\calO_L}(A,A^t)^{sym} \cong
\calO_L$, then principal $\calO_L$-polarisations are in bijection
with $\calO_L^{\times,+}$. Furthermore, under $h^+(L)=1$,
$\calO_L^{\times,+} = (\calO_L^{\times})^2$, so if $\lambda_1$ and
$\lambda_2$ are principal $\calO_L$-polarisations, $(A, \iota,
\lambda_1) \cong (A, \iota, \lambda_2)$ as polarized
$\calO_L$-abelian varieties, by multiplication by a suitable unit.
We point out that $h^+(L)=1$ by itself does not imply that all
abelian varieties with RM over $\Fp$ admit a principal
polarisation. On the other hand, a principally polarized abelian
variety with RM automatically satisfies the Rapoport condition.

We start our analysis by establishing a variant of Tate's theorem
on endomorphisms of abelian varieties taking into account the real
multiplications. This is well-known but we provide a proof for
completeness.

\begin{thm} \label{Tatethm}
Let $A_i, i = 1, 2$ be supersingular abelian varieties over $\Fp$
with RM by $\calO_L$ as above, $\iota_i: \calO_L \arr
\End_{\Fp}(A_i)$. Then for $\ell \neq p$,
$$\Hom_{\calO_L}(A_1,A_2) \otimes \Z_{\ell} \cong \Hom_{\calO_L\otimes \Z_\ell}(T_{\ell}(A_1),
T_{\ell}(A_2)),$$ and at $\ell = p$,
$$\Hom_{\calO_L}(A_1,A_2) \otimes \Z_p \cong \Hom_{\calO_L \otimes
W(k)[F,V]}(\D(A_1),\D(A_2)).$$
\end{thm}

\begin{proof}
We prove only the claim for $\ell \neq p$, as the proof is similar
for $\ell = p$. \noindent Choose a finite extension $\F_q\supseteq
\F_p$ such that all endomorphisms are defined over $\F_q$. We can
assume $A_1 = A_2 = A$ by an easy reduction (as in the proof of
the original theorem). Let $G := \Gal(\Fp/\F_q)$. Tate's theorem
(\cite{Tate}, \cite{ WaterhouseMilne}) implies that:
$$ \End(A/\F_q) \otimes \Z_{\ell} \cong
\End(T_{\ell}(A))^G.$$ Let $M_1 := \{ f \in \End(A/\F_q) \otimes
\Z_{\ell} | fr = rf, \forall r \in \calO_L \}$, which we may
identify with  $\{ f \in \End(T_{\ell}(A))^G | f T_{\ell}(r) =
T_{\ell}(r)f, \forall r \in \calO_L \}$, by Tate's theorem. Let
$M_2 = \{ f \in \End(A/\F_q) | fr = rf \} \otimes \Z_{\ell}.$ It
is clear that $M_2 \subseteq M_1$ and that $M_1/M_2$ is
torsion-free. By comparing ranks, the modules coincide. \noindent
We next use the fact that the endomorphism ring is large to
eliminate the mention of $G$ in the statement of the theorem. We
know that the action of $G$ commutes with the action of $\End(A)
\otimes_{\Z} \Q_{\ell}$ in $\End(T_{\ell}(A))\otimes_\Z \Q$. By
assumption, $A$ is supersingular and therefore $\End(A/\F_q)
\otimes \Q_{\ell} \cong M_{2g}(\Q_{\ell})$; the action of $G$ is
thus given by scalars in $\Q_\ell$ and thus $\End_{\calO_L}(A)
\otimes \Z_{\ell} \cong \End_{\calO_L}(T_{\ell}(A))$.
\end{proof}

\section{Classification of (some) Dieudonn\'e modules with multiplications}


Dieudonn\'e modules have been classified up to isomorphism by
Manin \cite{Manin}. For any fixed $g$, it turns out that there is
a unique superspecial Dieudonn\'e module of rank $2g$. We consider
here Dieudonn\'e modules with RM and impose some additional
conditions. These conditions are natural from the point of view of
abelian varieties. Note that we rely crucially on the arithmetic
assumption that the prime $p$ is unramified. For $p$ ramified in
$L$, the ideas involved are more subtle as there is more than one
isomorphism class of superspecial Dieudonn\'e modules with RM (see
\cite{these, NicoleRamified}).

Let $k$ be a perfect field of characteristic $p$. We concentrate
on rank $2g$ Dieudonn\'e modules with RM.

\begin{dfn}
 A Dieudonn\'e module $\D$ with RM is an $(\calO_L \otimes W(k))[F,V]$-module of rank~$2$ over $\calO_L
\otimes W(k)$, where $F,V$ commute with $\calO_L$.
\end{dfn}

Assume from now on that $k$ is algebraically closed. In order to
classify Dieudonn\'e modules with RM, we introduce useful
decompositions. A Dieudonn\'e module $\D$ with RM decomposes
according to
 the decomposition of $p$ in $\ol$:
$$ \D = \oplus_{\gerp \vert p} \D_\gerp,$$
where $\D_\gerp$ are Dieudonn\'e modules with
$\calO_{L_\gerp}$-action. Furthermore, letting $\Phi_\gerp =
\Hom(\ol_\gerp, W(k))$, each $\D_\gerp$ decomposes as an
$\ol\otimes W(k)$-module as $\oplus_{\phi\in \Phi_\gerp}
\D_\gerp(\phi)$, where $\D_\gerp(\phi)$ is the summand where
$\calO_{L_{\gerp}}$ acts via $\Phi_{\gerp}$. Altogether, we obtain
the decomposition
\[ \D \cong \oplus_{\gerp \vert p}\oplus_{\phi \in \Phi_\gerp}
W(k)_{\phi}^2, \] as $\ol\otimes W(k)$-modules.

\noindent Recall that an $\calO_L$-linear polarisation $A
\overset{\lambda}{\arr} A^t$ with degree prime to $p$ defines a
principal quasi-polarisation with RM i.e., a skew-symmetric form
$$\D(A) \times \D(A) \arr W(k) \otimes \calO_L,$$ which is a perfect
pairing over $W(k) \otimes \calO_L$, and such that $F$ and $V$ are
adjoint.

\noindent An important consequence of $p$ being unramified in $L$
is the following fact, whose use pervades the whole text. Recall
that a Dieudonn\'e module $\D$ is superspecial if $F\D = V\D$.

\begin{fact} \label{fact23}
Let $k$ be algebraically closed. Then there exists a unique
superspecial (principally quasi-polarized) Dieudonn\'e module with
RM.
\end{fact}

\begin{proof}
Recall that Dieudonn\'e modules decompose according to primes;
since $p$ in unramified, the claim follows from the inert case
(see \cite[Thm 5.4.4]{GO}).
\end{proof}

\noindent Moreover, since for any abelian variety $A$ over an
algebraically closed field $k$, $T_{\ell}(A) \cong (\calO_L
\otimes \Z_{\ell})^2$ for $\ell \neq p$ (cf. \cite[Lem.
1.3]{Rapoport}), the behaviour of $A$ at the prime $p$ is pivotal.

\noindent From the decomposition of $\D$ according to primes, we
may in the rest of this section assume that $p$ is inert.
\noindent The Dieudonn\'e module $\D$ is then a $\calO_{L_{\gerp}}
\otimes_{\Z_p} W(k)$-module of rank two, where
$\calO_{L_{\gerp_i}}$ is unramified of degree $g$. \noindent Put
$R = \calO_{L_{\gerp}}$. Up to a choice of an embedding $\tau$ of
$R$ in $W(k)$, we can decompose the ring:
$$R \otimes W(k) = \oplus_{\tau \in \Hom(R, W(k))} W(k)_{\tau} = \oplus_{i=1}^g
W(k)_i,$$
$$ r \otimes \lambda \mapsto \Big( \tau(r)\lambda, \sigma \circ
\tau(r) \lambda, \dots, \sigma^{g-1} \circ \tau(r) \lambda
\Big),$$ where $\sigma$ denotes the Frobenius. To get a uniform
notation, put $\sigma_i := \sigma^{i-1} \circ \tau$. Then $$\Big(
\tau(r)\lambda, \sigma \circ \tau(r) \lambda, \dots, \sigma^{g-1}
\circ \tau(r) \lambda \Big) = \Big( \sigma_1(r)\lambda,
\sigma_2(r)\lambda, \dots, \sigma_g(r)\lambda \Big).$$ In turn,
this implies that in the decomposition $\D = \oplus_{i=1}^{g}
\D_i$, the action of $R$ is given by: $$r \star d_i = \sigma^{i-1}
\circ \tau (r) \cdot d_i = \sigma_i(r) d_i,$$ for $d_i \in \D_i$,
$r \in R$.

\noindent Note that $F: \D_i \arr \D_{i+1}, V : \D_{i+1} \arr
\D_{i}$, $FV=VF=p$ and $$p\D_{i+1} = F(V\D_{i+1}) \subseteq
F(\D_i) \subseteq \D_{i+1} \quad \mbox{ for all } i .$$

\noindent There are some constraints on how these pieces fit
together. In particular, $$ \oplus_i \dim_k (\D_{i+1}/F\D_i) =
g,$$ with $\phi_i := \dim_k (\D_{i+1}/F\D_i) \in \{0,1,2\}$.

\begin{dfn} We call $\Phi = \{ \phi_i \}$ an admissible set of
indices if $\phi = 1$ for all $i$ or if $\phi \in \{ 0,2 \}$ for
all $i$.
\end{dfn}

\begin{prop} \label{classification} Let $\Phi = \{ \phi_i \}$ be an admissible set of indices.
Let $\D = \oplus_{i=1}^g \D_i$ be a supersingular Dieudonn\'e
module with RM with invariants $\Phi$. Then $\D$ is uniquely
determined e.g., by the elementary divisors of the
$\sigma^g$-linear map $F^g: \D_1 \arr \D_1$.
\end{prop}

\begin{proof}
\noindent We only give all the details of the proof supposing that
$\phi \in \{ 0,2 \}$ for all $i$. The other case occurs when the
Rapoport condition holds, and is shown in \cite[\S 3]{Stamm} to be
superspecial in which case the result follows by uniqueness (see
Fact \ref{fact23}).

\noindent Suppose that $\D'$ is a rank $2$ module over $W(k)$,
equipped with the action of a $\sigma^g$-linear operator that we
call $F^g$. By \cite[p.419]{Stamm}, if we pick a basis $\{ \alpha,
\beta \}$ for $\D'$, we can represent $F^g$ by a matrix $M$ with
elementary divisors i.e., according to whether $g$ is even or odd,
there exists a matrix $N \in \GL_2(W(k))$ such that
$$ N^{-1} M N^{\sigma^g} = \mat{p^\frac{g}{2}}{0}{0}{p^\frac{g}{2}} \mbox{ or } \mat
{0}{1}{p^g}{0}.$$

\noindent We now explain how to choose the bases in a normalized
way that is, starting from a basis $\{ \alpha_1, \beta_1 \}$ of
$\D_1$, we show how to define a basis for $\D_2$, and then for
$\D_3$, etc. back to the normalized map $F^g: \D_1 \arr \D_1$.

\noindent Consider the map $\D_i \overset{F}{\arr} \D_{i+1}$,
starting, say, with $i = 1$.
 \noindent If $\phi_{i+1} = 0$, then $F$ is an isomorphism, and for any basis
 $\{ \alpha_i, \beta_i \}$, define $\alpha_{i+1} := F(\alpha_i)$, and
 $\beta_{i+1} := F(\beta_i)$.
\noindent
 If $\phi_{i+1} = 2$, then $F$ is $p$ times an isomorphism, and for any basis
 $\{ \alpha_i, \beta_i \}$, define $\alpha_{i+1} := \frac{1}{p} F(\alpha_i)$, and
 $\frac{1}{p} \beta_{i+1} := F(\beta_i)$.
\noindent It follows from the construction that any change of
basis to $\{ \alpha_1, \beta_1 \}$ uniquely gives all $\ \{
\alpha_i, \beta_i \}$'s, and we conclude based on the uniqueness
of the normal presentation with elementary divisors.
\end{proof}

\section{Going global: Eichler orders}

\label{precedante}

In this section, we show how to construct abelian varieties with
real multiplication that are exotic from the point of view of
moduli. In particular, their endomorphism orders will be maximal
in $B_{p,L}$ e.g., level one may occur. Our starting point,
though, is the superspecial locus of the Hilbert moduli space.

We suppose that $A$ is a superspecial abelian variety with RM over
$k=\Fp$ satisfying the Rapoport condition i.e., the $\Lie(A)$ is a
free $\calO_L \otimes_{\Z} k$-module. Then, as we will see
shortly, $R = \Cent_{\End(A)}(\calO_L)$ is an Eichler order of
level $p$ in $B_{p,L}$. That is, $$ R \otimes_{\calO_L}
\calO_{L_{\gerp_i}} \cong
\begin{cases}  \mbox{ maximal order of } B_{p,L} \otimes
L_{\gerp_i} & f(\gerp_i/p) = 1 \mod{2} \\
 \mat{\ast}{\ast}{p\ast}{\ast} =: \Gamma_0(p)_{\gerp_i} \subseteq
M_2 (\calO_{\gerp_i}) & \mbox{ otherwise. } \end{cases}
$$

\noindent Recall that by local class field theory (see
\cite{Serre}), the condition $f(\gerp_i/p)$ odd implies that
$B_{p,L} \otimes L_{\gerp_i}$ is a division algebra.

\begin{prop} Let $A$ be a principally polarized superspecial abelian variety with RM over $\Fp$. The endomorphism order
$\End_{\calO_L}(A) := \Cent_{\End(A)}(\calO_L)$ is an Eichler
order of level $p$ in $B_{p,L}$.
\end{prop}

\begin{proof}
The $\calO_L$-version of Tate's theorem allows us to perform all
local computations on $T_{\ell}(A)$, which are uniquely determined
(this is clear for $\ell \neq p$, and at $\ell = p$, it follows
from Fact \ref{fact23}). For this local computation, pick $A = E
\otimes_{\Z} \calO_L$ for $E$ a supersingular elliptic curve. It
follows that $\End_{\calO_L}(A) = \End(E) \otimes \calO_L$ and the
level is clearly $p\calO_L$. Since $p$ is unramified, it is
squarefree, and thus the order is Eichler by \cite[Prop.
1.2]{Brz1}.
\end{proof}

From now on, we set $A = E \otimes_{\Z} \calO_L$; as we are
interested in local considerations in this section, this entails
no loss of generality and computations are more transparent. Note
that $\D(A) \cong \D(E) \otimes \calO_L$. \noindent For $\gerp_i$
above $p$ such that $f(\gerp_i/p) = 0 \mod{2}$, let $R_{\gerp_i}$
be the Eichler order $\Gamma_0(p)_{\gerp_i}$. Put $R^+_{\gerp_i}
:= M_2(\calO_{L_{\gerp_i}})$. \noindent We can associate to
$R^+_{\gerp_i}$ a (supersingular) abelian variety $A^+$ with RM
(depending on $\gerp_i$) and an isogeny $f: A \arr A^+$. The
isogeny $f$ (and thus $A^+$) is determined uniquely by $f^{\ast}
\D (A^+) $. We describe the Dieudonn\'e module more precisely:
$$ \D_{\gerp_i}(A^+) := \begin{cases} \D (E) \otimes_{\Z_p}
\calO_{L_{\gerp_i}}, & \gerp_i | p , f(\gerp_i/p) = 1 \mod{2} \\
p \cdot R^+_{\gerp_i} \cdot \Big( \D(E) \otimes_{\Z_p}
\calO_{L_{\gerp_i}}\Big) & \mbox{ otherwise. }
\end{cases} $$

\noindent Without loss of generality, we can assume that $p\calO_L
= \gerp$ is inert. Our goal is to describe $R^+_{\gerp} \cdot
(\D(E) \otimes_{\Z_p} \calO_{L_{\gerp}})$, in order to
characterize the abelian variety $A^+$ arising from such a
construction. We can thus suppose that $g$ is even, as it follows
from the condition: $f(\gerp) = 0 \mod{2}$.

\noindent The Dieudonn\'e module $\D(E)$ is identified with
$W(k)^2$, with
$F = (\begin{smallmatrix} 0 & 1 \\
p & 0
\end{smallmatrix})$ and $V = (\begin{smallmatrix} 0 & p \\
1 & 0
\end{smallmatrix})$. The maximal order $\calO = \End(E) \otimes_{\Z} \Z_p$ has the following
presentation: $ \calO = \{ (\begin{smallmatrix} a & b \\
pb^{\sigma} & a^{\sigma}
\end{smallmatrix})  : a,b \in W(\F_{p^2}) \}.$
With this notation, $R_{\gerp} \cong \calO \otimes_{\Z_p}
\calO_{L_{\gerp}}$.

\noindent Choose an isomorphism $\calO_{L_{\gerp}} \cong
W(\F_{p^g})$ (i.e., choose an embedding $\F_{p^2} \subseteq
\F_{p^g}$). We thus get an $\calO \otimes_{\Z_p}
W(\F_{p^g})$-action on $W(k)^2 \otimes_{\Z_p} W(\F_{p^g}) \cong
\oplus_{j=1}^g W(k)_j^2.$ On the other hand, $\calO \otimes_{\Z_p}
W(\F_{p^g}) = \Big( \calO \otimes_{\Z_p} W(\F_{p^2}) \Big)
\otimes_{W(\F_{p^2})} W(\F_{p^g})$ and note that $W(\F_{p^2})$
already splits $\calO \otimes \Q$.

\noindent For $p \neq 2$, one can explicitly pick a unit $b$ such
that $b^{\sigma} = -b$, and the element $$t :=
\mat{0}{b}{pb^{\sigma}}{0} \otimes 1/2pb^{\sigma} +
\mat{0}{\frac{1}{2p}}{\frac{1}{2}}{0} \otimes 1,$$ yields an
embedding $$\calO \otimes_{\Z_p} W(\F_{p^2}) \hookrightarrow
M_2(W(\F_{p^2})),$$ by sending $t \mapsto (\begin{smallmatrix} 0 & 0 \\
1 & 0 \end{smallmatrix})$. Under this embedding, the image of
$\calO \otimes_{\Z_p} W(\F_{p^2})$ is the chosen Eichler order of
level $p$, so that the lattice $R^+_{\gerp} \cdot \D(E) \otimes
\calO_L$ is:
$$ \{ \2vector{v_1}{w_1}, \2vector{v_2}{w_2}, \dots,
\2vector{v_g}{w_g} | v_i, w_i \in W(k) \mbox{ for all } i,j \mbox{
and } v_i \in 1/p \ \cdot W(k), \mbox{ for } i \mbox{ odd} \} .$$
\noindent There remains the task to compute explicitly the
Frobenius action. This follows from the previous presentation of
$R^+_{\gerp} \cdot \D(E) \otimes \calO_L$. Recall that the
Frobenius acts via the matrix $(\begin{smallmatrix} 0 & 1 \\
p & 0
\end{smallmatrix})$. The Frobenius $Fr \otimes 1 $ maps the bases for the various
$\D_i$ in the following way (with the notation $\Fr_i := \Fr
\otimes 1_{|W(k)_i}$):

$ i \mbox{ odd }
\begin{cases}
\begin{array}{lrlcl} \D_i & = & < \2vector{1/p}{0}, \2vector{0}{1}>
& \overset{\Fr_i}{\arr} & <\2vector{0}{1},
\2vector{1}{0}> \\
\D_{i+1} & = & <\2vector{1}{0}, \2vector{0}{1} > &
\overset{\Fr_{i+1}}{\arr} &< \2vector{0}{1}, \2vector{1}{0}>
\end{array} \\
\end{cases}
,$

\noindent and similarly for $i$ even i.e., $\Fr_i$ is again given by $(\begin{smallmatrix} 0 & 1 \\
p & 0
\end{smallmatrix})$.
 \noindent From this, we can compute the dimensions $\phi_i$
of the various quotients $\D_{i+1}/\Fr(\D_i) =
\D_{i+1}/p\D_{i+1}$. In short, $\phi_i =
\begin{cases}
0 & \quad i \mbox{ even } \\
2 & \quad i \mbox{ odd }
\end{cases}
$.

\noindent By Proposition \ref{classification}, we know that all
Dieudonn\'e modules with invariants $\{ \phi_i \}$ as above form a
unique isomorphism class. On the other hand, from the above
computations, the $a$-number of $A^+$ is easily seen to be $g$, so
$A^+$ is also a superspecial abelian variety. Note that it does
not satisfy the Rapoport condition. This is in a posteriori
agreement with the classification of superspecial crystals of
\cite{Yu2}.

\section{Left ideal classes and superspecial points}
\label{leftidealclasses}

From this section on, we suppose that $h^+(L)=1$; we recall that
any point on the Hilbert moduli space is thence principally
polarized. Abelian varieties and their morphisms are assumed to be
defined over $\Fp$, in particular, $\Hom_{\calO_L}(A_1,A_2) =
\Hom_{\calO_L, \Fp}(A_1,A_2)$ for $A_1,A_2$ abelian varieties with
RM.

We establish a link between the arithmetic of Eichler orders of
level $p$ in the quaternion algebra $B_{p,\infty} \otimes L$ and
the superspecial locus of the Hilbert moduli space.

\noindent The first step is to relate left ideals of an Eichler
order of level $p$ in $B_{p,L}$ and modules of $\calO_L$-isogenies
between superspecial abelian varieties. \noindent Let $\calO$ be
an order in a quaternion algebra. An $\calO$-ideal $I$ is
projective if and only if $I$ is locally principal (\cite[Prop.
1.1]{Brz1}). Thus, the following is a direct corollary of the
supersingular $\calO_L$-version of Tate's theorem (see Theorem
\ref{Tatethm}):

\begin{cor}
Let $A_1,A_2$ be two superspecial abelian varieties with RM
satisfying the Rapoport condition. Let $\calO =
\End_{\calO_L}(A_1)$. Then $\Hom_{\calO_L}(A_1,A_2)$ is a
projective $\calO$-module of rank one.
\end{cor}

\begin{proof}
Since $T_{\ell}(A_i) \cong (\calO_L \otimes \Z_{\ell})^2$ for
$\ell \neq p$, it follows from Thm.\ref{Tatethm} that
$$\Hom_{\calO_L}(A_1,A_2) \otimes \Z_{\ell} \cong \Hom_{\calO_L
\otimes \Z_{\ell}} (T_{\ell}(A_1), T_{\ell}(A_2)) \cong
\End_{\calO_L}(A_1) \otimes \Z_{\ell},$$ and for $\ell = p$, we
use Fact \ref{fact23}.
\end{proof}

Further, we show that this geometric construction recovers all
ideal classes by establishing a bijection between left ideals
classes and superspecial points. The classical proof using the
concept of kernel ideals (see \cite[Thm. 3.15]{Waterhouse})
applies equally well to hereditary orders (e.g., Eichler orders of
squarefree level), and thus can be extended to our setting (see
below for a definition of kernel ideals). Unfortunately, the
classical reference for the original result of Nehrkorn is:
\cite[Satz 27, p.106]{Deuring}, whose statement concerns maximal
orders. One has thus to check that the proof uses only the fact
that ideals of hereditary orders are locally principal and through
the implication that all locally principal ideals are kernel
ideals, one obtains the desired characterization. For this reason,
we indicate an alternate proof not relying on \cite{Deuring}.

Let $A$ be a superspecial abelian variety with RM satisfying the
Rapoport condition. Let $\calO = \End_{\calO_L}(A)$. We give an
explicit bijection by using the Serre tensor construction $A
\otimes_{\calO} -$ (see \cite[\S 7]{Conrad} for a description of
this formalism).

\begin{lem}
Let $A$ be a superspecial abelian variety with RM satisfying the
Rapo-port condition. Let $I$ be a projective rank one
$\calO$-module. Then $A \otimes_{\calO} I$ is also superspecial
abelian variety with RM satisfying the Rapoport condition.
\end{lem}

\begin{proof}
The tensor construction respects the Rapoport condition. Since $I$
is isomorphic to an ideal of $\calO$, it is always possible to
choose a representative that makes it locally isomorphic to
$\calO$ at any given prime. By Theorem \ref{Tatethm}, this implies
that the corresponding Tate modules (resp. Dieudonn\'e modules)
are isomorphic, and thus the abelian variety is superspecial.
\end{proof}

\noindent An argument bypassing Nehrkorn's theorem by using Tate's
theorem is provided by the sequence of three lemmas:

\begin{lem} \label{avoidingp} For any two superspecial abelian varieties with RM $A_1$ and $A_2$ satisfying the Rapoport condition, there exists an
$\calO_L$-isogeny $f: A_1 \arr A_2$ of degree prime to $p$.
\end{lem}

\begin{proof}
The idea it to approximate the isomorphism between the Dieudonn\'e
modules $p$-adically. All morphisms are $\calO_L$-morphisms, and
we drop the mention of $\calO_L$ from the notation. Choose $q$ big
enough so that $A_1,A_2$ are defined over $\F_q$. Fix an
isomorphism: $\phi \in \Hom(\D(A_1),
\D(A_2))^{\Gal(\overline{\F}_p/ \F_q)} \overset{\D^{-1}}{\cong}
\Hom_{\F_q}(A_1,A_2) \otimes \Z_p.$ Since $\Hom_{\F_q}(A_1,A_2)$
is a free $\Z$-module, $\Hom_{\F_q}(A_1,A_2)
\overset{dense}{\hookrightarrow} \Hom_{\F_q}(A_1,A_2)\otimes
\Z_{p}$ has dense image. Thus, there exists $f \in
\Hom_{\F_q}(A_1,A_2)$ such that $p^r | (\D(f) - \phi),$ and $p^r |
f - \D^{-1} (\phi))$. The degree of $\phi$ is equal to
$\lg_{\Z/p\Z} \D(A_2)/\phi(\D(A_1))$. Let us switch to matrix
representatives: Since $\phi$ is an isomorphism (i.e.,
$\deg_p(\phi) = 1$), the determinant of its matrix $M(\phi) \in
M_{2g}(W(\overline{\F}_p))$ is a unit. Call $N_f$ the matrix
representative of $f$. Since $p^r | (N_f - M(\phi))$, it follows
that $\det(N_f)$ is a $p$-adic unit if we take $r$ big enough, and
thus $\deg_p(f) = 1$.
\end{proof}

\noindent We can associate an $\calO_L$-subgroup scheme to an
ideal $J$ by taking the scheme-theoretic intersection $A[J] :=
\cap_{f \in J} A[f]$. Let $I(H)$ be the ideal of $\calO$
associated to a finite $\calO_L$-group scheme $H$ i.e., $$I(H) :=
\{ f \in \calO | f(H) = 0 \}.$$ We say that an ideal $J$ is a
kernel ideal if $I(A[J]) = J$.

\begin{lem} Every ideal $J$ of $\calO$ of prime-to-$p$ norm  is a kernel
ideal i.e., $J = I(A[J])$.
\end{lem}

\begin{proof}
Under the prime-to-$p$ condition, this is a straightforward
computation with Tate modules. Since this is meant to replace the
appeal to Nehrkorn's theorem, we give details. \noindent First,
note that $\tilde{J} = I(A[J])$ is an ideal. We compute locally:
we need to show that $\tilde{J_{\gerq}} = J_{\gerq}$ for $\gerq
\not| \ p$, and $\tilde{J_{\gerq}} = \calO_{\gerq}$ for $\gerq |
p$. Since $J$ is an integral ideal, we can fix isomorphisms
$\calO_{\gerq} \cong M_2(\calO_{L_{\gerq}})$ for $\gerq \not| \
p$, and thus since $J$ is locally principal, there are $m_q \in
M_2(\calO_{L_{\gerq}})$ such that $J_{\gerq} = \calO_{\gerq}
m_{\gerq}$. Second, write
$$I(A[J])_{\gerq} = \tilde{J} \otimes_{\calO_L} \calO_{L_{\gerq}}
\subseteq \calO \otimes_{\calO_L} \calO_{L_{\gerq}} =
\End_{\calO_{L_{\gerq}}}(T_{\gerq}(A)),$$ where $T_{\gerq}(A)$ is
defined by replacing $\ell$ by $\gerq$ in the definition of the
Tate module.

 \noindent Note
that for $H = A[J]$, $I(A[J])_{\gerq} = \{ f \in
\End_{\calO_{L_{\gerq}}}(T_{\gerq}(A)) | f(T_{\gerq}(A)) \subseteq
\Lambda_{\gerq}(H) \}$, where $\Lambda_{\gerq}(H) :=  \pi_H^{-1}
(T_{\gerq}(A/H) )/ T_{\gerq} (A) $, for $\pi_H : T_{\gerq}(A) \arr
T_{\gerq}(A/H)$. Next, we check that $A[J]_{\gerq} \cong T_{\gerq}
(A) / \sum_{f \in J_{\gerq}} f(T_{\gerq}(A)) = T_{\gerq}(A) / m_q
T_{\gerq}(A)$ for $\gerq \not| \ p$. Using the identifications
$\calO_{\gerq} \cong M_2(\calO_{L_{\gerq}})$, we check easily that
if $n_{\gerq} \in \GL_2(\calO_{L_{\gerq}})$ and $n_{\gerq}
T_{\gerq}(A) \subseteq m_q T_{\gerq}(A)$, then $n_q \in M_2
(\calO_{L_{\gerq}}) m_q$. This implies that $\tilde{J_{\gerq}}
\subseteq J_{\gerq}$ and thus $\tilde{J_{\gerq}} = J_{\gerq}$. To
finish the proof, we need to check the remaining equality at
$\gerq | p$: since $A[J]_{\gerq} = \{ 0 \}$, $J_{\gerq} =
\calO_{\gerq}$, we get that:
$$ \tilde{J_{\gerq}} = \{ f \in
\End_{\calO_{L_{\gerq}}}(T_{\gerq}(A) | f(T_{\gerq}(A)) \subseteq
T_{\gerq}(A) \} = \calO_{\gerq}.$$

\end{proof}

\begin{lem}
Any $\calO_L$-subgroup $H$ of $A$ of prime-to-$p$ order is of the
form $A[I]$ for some integral ideal $I$ of $\calO$.
\end{lem}

\begin{proof}
This is clear, as any sublattice (of finite index) in
$\calO_{L_{\gerq}}^2$ has the form $m_{\gerq} \calO_{L_{\gerq}}^2$
for some $m_q \in M_2(\calO_{L_{\gerq}})$, $ \gerq \not| \ p$.
\end{proof}

\noindent Note that $A/A[I] \cong \Hom_{\calO}(I,A) \cong A
\otimes_{\calO} I^{-1}$.

\begin{thm} \label{bijection} Let $h^+(L)=1$. Let $A$ be a
superspecial abelian variety with RM satisfying the Rapoport
condition. The map $I \mapsto A \otimes_{\calO} I$ induces a
functorial bijection from left ideal classes of $\calO$ to
superspecial abelian varieties with RM satisfying the Rapoport
condition.
\end{thm}

\begin{proof}
Since all necessary ideals are kernel ideals, the map is injective
by \cite[Prop. 3.11]{Waterhouse}. Functoriality follows from the
definition of the tensor construction. Surjectivity follows from
the general formalism: let $A'$ be another superspecial abelian
variety with RM. Then the natural map:
$$ A \otimes_{\calO} \Hom_{\calO_L}(A,A') \overset{\psi}{\arr} A',
\quad a \otimes \phi \mapsto \phi(a),$$ is a well-defined
isomorphism of abelian varieties with RM, and considering the
projective $\calO$-module $\Hom_{\calO_L}(A,A')$ gives the desired
preimage of $A'$.
\end{proof}

\begin{cor}
All Eichler orders of level $p$ in the quaternion algebra
$B_{p,L}$ arise from geometry.
\end{cor}

\begin{proof}
Since Eichler orders of level $p$ are locally isomorphic
(\cite[Prop. 5.3]{Brz1}), the set of right orders of a complete
set of representatives of left, projective ideal classes of any
Eichler order of level $p$ represent all isomorphism classes by
\cite[Lem. 4.10, p.26]{Vigneras} and \cite[Cor. 5.5,
p.88]{Vigneras}. By Theorem \ref{bijection}, it is enough to
consider the right orders $\End_{\calO_L}(A_m)$ of
$\Hom_{\calO_L}(A,A_m)$ for varying $m$.
\end{proof}

The next step is to show that the quadratic module structure
coming from the norm of the quaternion algebra can be defined
geometrically. As for supersingular elliptic curves, a necessary
ingredient is the existence of two isogenies of coprime degree
between any two superspecial points.

Let $A$ be a fixed principally polarized superspecial abelian
variety with RM. Let $G$ be the group scheme over $\Spec(\Z)$
whose group of $R$-points, for any commutative ring $R$, is:
$$G(R) = \{ \phi \in (\End_{\calO_L}(A) \otimes R)^{\times} |
\phi' \phi = 1 \},$$ where $\phi \mapsto \phi'$ is the Rosati
involution induced by the polarisation of $A$.

\noindent Using Fact \ref{fact23}, the superspecial locus
$\Lambda$ is parametrized by double cosets by a theorem of Yu
(\cite[Thm. 10.5]{Yu}). More precisely, the set $\Lambda$ is in
natural bijection with the adelic double cosets $G(\Q) \backslash
G(\A_f) / G(\widehat{\Z})$. As in the elliptic case, a standard
application of the strong approximation theorem (cf.
\cite[p.81]{Vigneras}) then shows that for $\ell \neq p$, the
$\ell$-power Hecke orbit of a superspecial point on the Hilbert
moduli space is the whole superspecial locus (giving incidentally
a stronger result than Lemma \ref{avoidingp}).
\noindent Summing
up, we get the desired:

\begin{cor} \label{strongapprox}
Let $A_1,A_2$ be two principally polarized superspecial abelian
varieties with RM. Then for any prime $\ell \neq p$, there exists
an $\ell$-power isogeny between $A_1$ and $A_2$. In particular,
the module $\Hom_{\calO_L}(A_1,A_2)$ contains two isogenies which
are of coprime degrees.
\end{cor}

There remains to give to the module $\Hom_{\calO_L}(A_1,A_2)$ a
quadratic structure by defining an associated quadratic form
geometrically. The presence of polarisations and the totally real
field $L$ introduces some ambiguity, so we give details.

\noindent Let $A_1,A_2$ be supersingular abelian varieties with
RM. Let $\lambda_i : A_i \overset{\cong}{\arr} A_i^{t}, i=1,2,$ be
principal $\calO_L$-polarisations, and define, for $\phi \in
\Hom_{\calO_L}(A_1,A_2)$ $$||\phi|| = || \phi ||_{\lambda_1,
\lambda_2} := \lambda_1^{-1} \circ \phi^{t} \circ \lambda_2 \circ
\phi, \qquad
\begin{CD}
{A_2^t} @< {\lambda_2} << {A_2} \\
@ V{\phi^t} VV @ AA {\phi} A \qquad .\\
{A_1^t} @>> {\lambda_1^{-1}} > {A_1} \\
\end{CD}
$$

\noindent Then we obtain a function which we call the
$\calO_L$-degree:
$$|| - ||: \Hom_{\calO_L}(A_1,A_2) \arr \End_{\calO_L}(A_1).$$

\begin{lem} \label{oldegree}
The $\calO_L$-degree $||-||$ takes values in $\calO_L$ and is a
totally positive $\calO_L$-integral quadratic form i.e.,
\begin{enumerate}
\item $||\phi|| = 0 \mbox{ if and only if } \phi = 0 $ and
$||\phi|| \gg 0$ i.e., it is totally positive for all $\phi \neq
0$; \item $<\phi,\psi>_{\calO_L} := || \phi + \psi || - ||\phi|| -
||\psi|| =  \lambda_1^{-1} \psi^t \lambda_2 \phi +
\lambda_1^{-1}\phi^t \lambda_2 \psi, $ is a symmetric
$\calO_L$-bilinear form; \item $|| \ell \circ \phi || = \ell^2 ||
\phi ||, \mbox{ for } \ell \in \calO_L $.
\end{enumerate}
\end{lem}

\begin{proof}
The element $|| \phi ||$ is fixed by the Rosati involution $f
\mapsto f' = \lambda_1^{-1} f^{t} \lambda_1$ :
$$ \lambda_1^{-1} \cdot (\lambda_1^{-1} \circ \phi^t \circ \lambda_2 \circ \phi)^t \lambda_1 =  \lambda_1^{-1} \circ \phi^t \circ \lambda_2 \circ \phi .$$
The Rosati involution fixes $L$ in $\End_{\calO_L}^0(A_1)$ that
is, if $A_1$ and $A_2$ are supersingular abelian varieties, it
follows from Albert's classification that we are in the Type III
situation: the quaternion algebra $\End_{\calO_L}(A_1) \otimes \Q$
over the totally real field $L$ is totally definite, hence the
Rosati involution is the canonical involution i.e., the
conjugation map i.e., $x^{\sigma} = \Tr(x) - x = \overline{x}$ on
the quaternion algebra $B_{p,L}$. Since $\lambda_1$ is principal,
all computations are integral, and the image of $||-||$ is
$\calO_L$.

\noindent Let us check Assertion $(1)$. Clearly, $||\phi|| = 0$ if and only if $\phi$ is the zero map (any non-zero $\calO_L$-homomorphism of abelian varieties is an isogeny). 
The total positivity follows from properties of the embedding of
the N\'eron-Severi group $NS^0(A)$ in
$\End_{\calO_L}^0(A_1)^{sym}$ via the map $\mu \mapsto
\lambda_1^{-1} \mu$: the polarisations are sent to positive
symmetric elements. The remaining claims are straightforward
computations.
\end{proof}

\noindent It is easy to see that the $\calO_L$-degree $||-||$ is
multiplicative: if $\psi \in \Hom_{\calO_L}(A_2,A_3)$ and $\phi
\in \Hom_{\calO_L}(A_1,A_2)$, then $ || \psi \circ \phi
||_{\lambda_1,\lambda_3} = || \psi ||_{\lambda_2,\lambda_3} \cdot
|| \phi ||_{\lambda_1,\lambda_2},$ and this property defines the
quadratic form up to a constant multiple.

\begin{lem} \label{112}
Suppose that $A_1$ and $A_2$ are superspecial. Let $\psi \in
\Hom_{\calO_L}(A_1,A_2)$. Let $\calO = \End_{\calO_L}(A_1)$. We
can use $\psi$ to embed $\Hom_{\calO_L}(A_1,A_2)$ as an
$\calO$-ideal:
$$ \Hom_{\calO_L}(A_1,A_2) \overset{j_{\psi}}{\hookrightarrow} \End_{\calO_L}(A_1),$$
$$ \phi \mapsto \lambda_1^{-1} \circ \psi^t \circ \lambda_2 \circ \phi.$$
Then $$ \Norm(j_{\psi}(\phi)) = ||\psi^t|| \cdot ||\phi||.$$ Let
$I_{\psi}$ be the $\calO$-ideal
$j_{\psi}(\Hom_{\calO_L}(A_1,A_2))$. Then the reduced norm
$N(I_{\psi})$ is equal to the ideal $(||\psi^t||)$.
\end{lem}

\begin{proof}
Compute the norm of $j_{\psi}(\phi)$:
$$
\begin{array}{ll}
\Norm(j_{\psi}(\phi)) &= \overline{j_{\psi}(\phi)}{j_{\psi}(\phi)} \\
&= [\lambda_1^{-1} \circ (\lambda_1^{-1} \circ \psi^t \circ \lambda_2 \circ \phi)^t \circ \lambda_1] \circ [\lambda_1^{-1} \circ \psi^t \circ \lambda_2 \circ \phi] \\
 &= [\lambda_1^{-1} \phi^t \lambda_2 \psi] [\lambda_1^{-1} \psi^t \lambda_2 \phi] \\
&= [\lambda_2 \psi \lambda_1^{-1} \psi^t] [\lambda_1^{-1} \phi^t \lambda_2 \phi] \\
&= ||\psi^t|| \cdot ||\phi||,
\end{array}
$$
\noindent since $\lambda_2 \psi \lambda_1^{-1} \psi^t \in
\calO_L$. It follows that the norm of $I_{\psi}$, being the
greatest common denominator of the norms of the elements
$||j_{\psi}(\phi)||$, is the greatest common denominator of all
$||\psi^t|| \cdot ||\phi||$, for $\phi \in
\Hom_{\calO_L}(A_1,A_2)$. Any two superspecial abelian varieties
admit two isogenies $\phi_1, \phi_2$ of relatively coprime degree
by Corollary \ref{strongapprox}, thence it follows that
$N(I_{\psi}) = (|| \psi^t ||).$

\end{proof}
\noindent Summarizing the previous discussion, we obtain the
desired link between the norm map and the $\calO_L$-degree.

\begin{prop} \label{indeterminacy}
Let $A_1,A_2$ be principally polarized abelian varieties with RM.
Let $\calO := \End_{\calO_L}(A_1)$. Let $\Hom_{\calO_L}(A_1,A_2)
\cong I$, $I$ an integral $\calO$-ideal. Let $\phi_x \in
\Hom_{\calO_L}(A_1,A_2)$ map to $x \in I$. Then, up to a unit, the
following formula holds:
$$ ||\phi_x|| = \frac{\Norm(x)}{\Norm(I)} .$$
\end{prop}

\noindent The indeterminacy between the $\calO_L$-degree and the
norm of an {\em element} is thus a totally positive element
well-defined up to a totally positive unit.

We recall the classical description of the theta series associated
to a quadratic module. As we have seen, the $\calO_L$-module
$\Hom_{\calO_L}(A_1,A_2)$ becomes a quadratic module when equipped
with the $\calO_L$-degree. We associate to this quadratic module a
theta series by the recipe:
$$\theta_{\Hom_{\calO_L}(A_1,A_2)} := \sum_{\calO_L \ni \nu \gg 0 \mbox{ or } \nu = 0}  a_{\nu} q^{\nu},$$where $$a_{\nu} = |\left\{ \phi \in \Hom_{\calO_L} (A_1,A_2) \mbox{ such that } || \phi || = \nu
\right\}|.$$

\begin{thm} A theta series constructed from a quadratic $\calO_L$-lattice $(M,q)$ of level $\calN$ yields a Hilbert modular form of weight $2$ and quadratic character $\chi_{M}$ (modulo the level) given by a Gauss sum,
which transforms under the group $$\SL_2(\calO_L \oplus \calN
\cdot \Norm(M) \calD_L)$$ defined as $\{ \left( \begin{smallmatrix} a & b \\
c & d \\ \end{smallmatrix} \right) \in \left( \begin{smallmatrix}
{\calO_L} & {(\Norm(M)\calD_L)^{-1}} \\ {\calN \Norm(M) \calD_L} &
{\calO_L} \\ \end{smallmatrix} \right) | ad-bc = 1 \}, $ where
$\calN$ is the level of the lattice $M$, and $\calD_L$ is the
different of $L$.
\end{thm}

\begin{proof}
See \cite[Th. I]{Eichler}.
\end{proof}

\noindent By specializing this result, we get the desired
description of our theta series:

\begin{cor} Let $h^{+}(L)=1$. Let $A_1,A_2$ be two principally polarized superspecial abelian varieties
with RM. Then the theta series $\theta_{\Hom_{\calO_L}(A_1,A_2)}$
is a Hilbert modular form of parallel weight $2$ and level
$\Gamma_0((p))$ with trivial quadratic character.
\end{cor}

All results of this section admit analogous versions for levels
dividing $p$, with the difference that the corresponding abelian
varieties do not satisfy Rapoport's condition. The essential point
is that the endomorphism orders are also Eichler. We make this
more precise for level one i.e., maximal orders. Let $A^{+}$ be a
superspecial abelian variety with RM such that $\calO^{+} :=
\End_{\calO_L}(A^{+})$ is a maximal order of level one in
$B_{\infty_1, \dots, \infty_g} = B_{p,\infty} \otimes L$, as
constructed in Section \ref{precedante}. For short, we call such a
superspecial abelian variety with RM an abelian variety of level
one.

\begin{prop} Let $h^+(L)=1$. The map $I \mapsto A^{+} \otimes_{\calO^{+}} I$ induces a
functorial bijection from left ideal classes of $\calO^{+}$ to
isomorphism classes of abelian varieties of level one.
\end{prop}

\begin{cor} Any maximal order of level one in $B_{p,\infty} \otimes L$ arises from geometry as the
endomorphism order of an abelian variety of level one.
\end{cor}

\noindent Moreover, the construction of theta series from two
abelian varieties of level one goes through in the same way as in
level $p$.

\begin{cor} \label{levelonetheta} Let $h^{+}(L)=1$. Let $A_1,A_2$ be two abelian varieties
of level one. Then the theta series
$\theta_{\Hom_{\calO_L}(A_1,A_2)}$ is a Hilbert modular form of
parallel weight $2$ and level $\Gamma_0(1)$ with trivial quadratic
character.
\end{cor}

\section{The Eichler Basis Problem}

The Jacquet-Langlands correspondence, as a special case of the
theta correspondence, provides in a sense a solution to the Basis
Problem. In classical terms, this translates in the statement that
the subspace of Hilbert modular newforms can be embedded in a
space of automorphic forms spanned by suitable theta series. In
the setting of this paper, we want theta series coming from
Eichler orders. Fortunately, this case follows readily from the
literature (cf. \cite[Chap. 9]{HPS}, where are mentioned thornier
issues arising with other (special) orders). \noindent We sketch
the argument. We write $S^{B_{p,L}}_2(\gerN,\C)$ for the space of
quaternionic modular forms on $B_{p,L}$ of level $\Gamma_0(\gerN)$
(as in \cite[p.201]{Hidabook}). The space
$S_2(\Gamma_0(\gerN),\C)$ is the space of classical Hilbert cusp
forms of weight $2$ of level $\Gamma_0(\gerN)$.

\begin{thm} \emph{(\cite[Thm. 16.1]{JL})}
Suppose that $\gerN = \gerN_0 d(B_{p,L})$ for an integral ideal
$\gerN_0$ prime to $d(B_{p,L})$. Then we have a Hecke-equivariant
embedding $S^{B_{p,L}}_2(\gerN,\C) \hookrightarrow
S_2(\Gamma_0(\gerN),\C)$. The image of this embedding consists of
cusp forms in $S_2(\Gamma_0(\gerN),\C)$ new at all primes dividing
$d(B_{p,L})$.
\end{thm}

\noindent It follows that if we let $\calH^{\gerN}_{\C}$ be the
prime-to-$\gerN$ Hecke algebra in $\End(S_2^{B_{p,L}}(\gerN,\C))$,
for $\gerN$ squarefree, then $S_2^{B_{p,L}}(\gerN,\C)$ is free of
rank $1$ over $\calH^{\gerN}_{\C}$. We can define a map $\Theta$
via classical theta series associated to Eichler orders (see
\cite[Eq. 7.9]{Hidabook}, cf. \cite[p.294]{Gelbart2}, \cite[\S
10]{Gelbart}) : $$\Theta: S_2^{B_{p,L}}(\gerN,\C)
\otimes_{\calH(\C)} S_2^{B_{p,L}}(\gerN,\C) \arr
S_2^{new}(\Gamma_0(\gerN);\C), \quad (f,g) \mapsto \Theta(f,g),$$
which is an isomorphism since $S_2^{B_{p,L}}(\gerN,\C)$ is free of
rank $1$. We state our conclusion.

\begin{cor} \label{basisEichler} Let $\gerN$ be squarefree.
The space $S^{new}_2(\Gamma_0(\gerN))$ is contained in the span of
theta series coming from left ideals of an Eichler order of level
$\gerN$ in the quaternion algebra $B_{p, \infty} \otimes L$.
\end{cor}

\medskip\noindent
{\bf Acknowledgments.} We would like to thank E.Z. Goren for
introducing us to the Hilbert modular world, and in particular for
bringing to our attention the possibility of level one when the
prime $p$ is unramified. The author was supported by a
postdoctoral fellowship and \emph{kaken-hi} (grant-in-aid) of the
Japanese Society for the Promotion of Science (JSPS) while working
on this paper at the University of T\=oky\=o.

\

\hbox{Email: nicole@ms.u-tokyo.ac.jp} \hbox{Address:  Department
of Mathematical Sciences, University of T\=oky\=o,} \hbox{Komaba,
153-8914, T\=oky\=o, Japan.}

\end{document}